\documentclass[11pt]{amsart}
\usepackage[left=1in,right=1in]{geometry}
\usepackage{amssymb,amsmath,amsthm}
\usepackage[flushmargin]{footmisc}
\usepackage[foot]{amsaddr}
\usepackage{wasysym}
\usepackage{graphicx}
\usepackage{epsfig,subfig}
\usepackage{color}
\usepackage[nocompress]{cite}

\title{Nonsmooth Homoclinic Bifurcation in a Conceptual Climate Model}

\author{Julie Leifeld}

\begin{document}

\newtheorem{theorem}{Theorem}
\newcommand{\ode}[2]{\dfrac{d{#1}}{d{#2}}}
\newcommand{\pde}[2]{\frac{\partial{#1}}{\partial{#2}}}
\newcommand{\red}[1]{\textcolor{red}{#1}}
\newcommand{\sgn}{\mathop{\mathrm{sgn}}}
\newcommand{\Tr}{\text{Tr}}
\newcommand{\ep}{\varepsilon}

\begin{abstract}

Collision of equilibria with a splitting manifold has been locally studied, but might also be a contributing factor to global bifurcations.  In particular a boundary collision can be coincident with collision of a virtual equilibrium with a periodic orbit, giving an analogue to a homoclinic bifurcation.  This type of bifurcation is demonstrated in a nonsmooth climate application.  Here we describe the nonsmooth bifurcation structure, as well as the smooth bifurcation structure for which the nonsmooth homoclinic bifurcation is a limiting case.

\end{abstract}

\maketitle

\section{Introduction}
\label{sec:intro}

Recently a lot of work has gone into classifying nonsmooth bifurcations.  In particular, bifurcations in two dimensional piecewise smooth systems are believed to be well understood, and have been described in \cite{Colombo12, diBernardo, Guardia11, Kuznetsov03, makarenkov, filippov}.  Here, we look at a particular type of piecewise smooth bifurcation, called {\it{border collision}}.  This bifurcation occurs when an equilibrium point collides with a curve of discontinuity, or {\it{splitting manifold}}.  If the equilibrium is a node, the existing literature has focussed on whether the equilibrium persists or is annihilated by the collision.  However, Filippov analysis of a conceptual climate model, called Welander's model, shows that there are other interesting questions to be asked about border collision bifurcations.  In some cases, the node may persist but change stability, implying a fundamental change in the behavior of the system.  In this case, it is necessary to explore the global dynamics of the system to gain a full understanding.  In Welander's model this change in stability occurs simultaneously to a homoclinic explosion, in which infinitely many homoclinic orbits originate from the boundary equilibrium.  These homoclinic orbits are also not of similar ilk to many other homoclinic orbits seen in nonsmooth bifurcation classification, as they do not involve grazing of periodic orbits with splitting manifolds \cite{diBernardo,Guardia11,Kuznetsov03,makarenkov}.  As the system moves through the bifurcation, the homoclinic orbits are instantaneously destroyed, leaving behind a large periodic orbit.  After briefly describing some necessary vocabulary in Section \ref{sec:nonsmooth}, the climate model will be introduced in Section \ref{sec:wel}, and both a local and global analysis of the model will be discussed in Sections \ref{sec:collision} and \ref{sec:global}.  Finally, a description of the smooth phenomena which limits to the nonsmooth bifurcation will be undertaken in Section \ref{sec:smooth}.

\section{Filippov Analysis of Nonsmooth Systems}
\label{sec:nonsmooth}

It is necessary to include a brief note about notation.  A piecewise smooth system can be written in the form
\begin{equation}
\dot{x}=f(x,\lambda)
\end{equation}
with $f$ a vector function depending on the variable $x$ and a nonsmooth function $\lambda$,
\begin{equation}
\label{eq:lambda}
\lambda=\left\{\begin{array}{cl} 1 & h(x)>0 \\ 0 & h(x)<0.\end{array}\right.
\end{equation}
$h(x)$ is a scalar function which defines the location of the discontinuity curve, or {\it{splitting manifold}}.

Behavior of the system away from the splitting manifold can be determined using classical methods, but along the splitting manifold new dynamics must be defined.  Here we use Filippov's convex combination method \cite{filippov}.  If the vector fields on either side of the splitting manifold point in the same direction, one expects trajectories there to cross, and we call this a {\it{crossing region}}.  Alternatively, if both vector fields point either toward or away from the splitting manifold, trajectories slide along the manifold, and we call this a {\it{sliding region}}.  Sliding regions exist when the system of equations
\begin{equation}
\label{eq:s}
\begin{array}{r}S=f(x;\lambda)\cdot\nabla h(x)=0\\
h(x)=0. \end{array}
\end{equation}
can be solved for $\lambda\in[0,1]$.  When $f(x,\lambda)$ has a linear dependence on $\lambda$, this corresponds to locations for which the convex combination of the vector fields contains a vector in the direction of the splitting manifold.  Stability of the splitting manifold itself can be found using \eqref{eq:s}.  If
\[
\ode{}{\lambda} S<0
\]
the sliding region is stable, and if
\[
\ode{}{\lambda} S>0
\]
the sliding region is unstable.

Sliding flow along the splitting manifold can be determined by solving \eqref{eq:s} for $\lambda^*(x)$, then looking at the dynamics of
\begin{equation}
\label{eq:slidingflow}
\dot{x}=f(x,\lambda^*(x)).
\end{equation}
Dynamical phenomena from \eqref{eq:slidingflow} can then be determined in the normal way.  Equilibria of equation \eqref{eq:slidingflow} are called {\it{pseudoequilibria}}.  These pseudoequilibria are {\it{pseudonodes}} if they are nodes in the Filippov flow and their Filippov stability agrees with the stability of the sliding region.  Otherwise they are called {\it{pseudosaddles}}.

\section{The Conceptual Climate Model}
\label{sec:wel}
Welander's Model \cite{Welander82} is an ocean convection model, given here in nondimensionalized form,
\begin{equation}
\label{eq:Welandersystem}
\begin{array}{c}\begin{array}{rcl} \dot{T}&=&1-T-k(\rho)T\\ \dot{S}&=&\beta(1-S)-k(\rho)S, \end{array} \\[8pt]
 \\
\rho=-\alpha T+S.
\end{array}
\end{equation}
$T$ and $S$ are temperature and salinity of the surface ocean, and $\rho$ is water density, given by a linear combination of the variables.  $k(\rho)$ is a convective mixing function, which is assumed to be small when $\rho$ is small, and large when $\rho$ is large.  Welander gives two specific functions for $k(\rho)$, 
\begin{equation}
\label{eq:ksmooth}
k(\rho)=\frac{1}{\pi}\tan^{-1}\left(\frac{\rho-\varepsilon}{a}\right)+\frac{1}{2},
\end{equation}
and the pointwise limit,
\begin{equation}
\label{eq:knonsmooth}
k(\rho)=\left\{\begin{array}{rl}1 & \rho>\varepsilon \\ 0 &\rho<\varepsilon. \end{array}\right.
\end{equation}
Here, we focus on the nonsmooth model, although the corresponding smooth model will be discussed in Section \ref{sec:smooth}.

Welander's Model has three parameters, $\alpha$, $\beta$, and $\ep$.  We will use Welander's values for $\alpha$ and $\beta$, 
\begin{equation}
\label{alpha}
\alpha=\dfrac{4}{5},
\end{equation}
and
\begin{equation}
\label{beta}
\beta=\dfrac{1}{2}.
\end{equation}
$\ep$ will be the bifurcation parameter.  Because the location of the discontinuity in equation \eqref{eq:knonsmooth} is a function of $\ep$, we will also do a preliminary coordinate change, 
\[
\begin{array}{rcl} x&=&T\\
y&=&S-\alpha T-\varepsilon, \end{array}
\]
giving a new system
\begin{equation}
\label{eq:Welcc}
\begin{array}{l}
\begin{array}{rcl}\dot{x}&=&1-x-k(y)x\\
\dot{y}&=&\beta-\beta\varepsilon-k(y)\varepsilon-\alpha-(\beta+k(y))y-(\alpha\beta-\alpha)x
\end{array}\\
 \\
k(y)=\left\{\begin{array}{rl}1 & y>0\\ 0 & y<0\end{array}\right. .
\end{array}
\end{equation}
Because this coordinate change is linear, it causes no difficulties in the Filippov analysis of the system.  In this case, $k$ is of the form of equation \eqref{eq:lambda} in Section \ref{sec:nonsmooth}.  Welander noted that for small, negative $\ep$, the nonsmooth system contains two stable virtual equilibria, whose stability is parameter independent.  Attraction to these virtual equilibria combined with switching across the splitting manifold give rise to a periodic orbit for certain values of $\ep$.  However, this periodic orbit does not exist for all values of $\ep$.  At $\ep=0$, a {\it{fused focus bifurcation}} occurs, in which the sliding region changes stability, birthing a stable periodic orbit.  One might also expect some bifurcation structure pertaining to collision of a virtual equilibrium with the splitting manifold.  In particular, if one of the virtual equilibria in Welander's model crosses the splitting manifold, it should become a global attractor, influencing the global dynamics.  In fact this does occur in an $\ep$ range, in a bifurcation which destroys the periodic orbit.  We will look at this bifurcation locally and globally in the next sections.

\section{Border Collision in the Model}
\label{sec:collision}

Oscillation in Welander's model is predicated on attraction to two different virtual equilibria.  However, it is immediately clear that parameter changes can transition these virtual equilibria into globally stable, real equilibria.  The location of the virtual equilibria are given by the system of equations
\begin{equation}
\label{eqloc}
\begin{array}{rcl}
x&=&\dfrac{1}{1+k}\\[8pt]
y&=&\dfrac{\beta-(\beta+k)\varepsilon-\alpha-(\alpha\beta-\alpha)x}{\beta+k}.
\end{array}
\end{equation}
This gives, for $k=1$, 
\begin{equation}
\label{eq:eq1}
(x,y)=\left(\dfrac{1}{2},-\dfrac{1}{15}-\varepsilon\right),
\end{equation}
and for $k=0$, 
\begin{equation}
\label{eq:eq2}
(x,y)=\left(1,\dfrac{1}{5}-\varepsilon\right).
\end{equation}

The $x$ locations of the equilibria are fixed, but the equilibria are located on the splitting manifold for 
\begin{equation}
\label{eq:v1}
\varepsilon=\dfrac{1}{5}
\end{equation}
and
\begin{equation}
\label{eq:v2}
\varepsilon=-\dfrac{1}{15}.
\end{equation}
When 
\[
0<\varepsilon<\dfrac{1}{5},
\]
solutions are attracted to a globally stable pseudoequilibrium.  So, although the equilibrium boundary collision is interesting, it does not qualitatively change the dynamics of the system.  Instead, the equilibrium escapes the splitting manifold and becomes a more classically recognizable, real equilibrium.  

However, there is a fused focus Hopf bifurcation which occurs when $\varepsilon=0$.  The stable Filippov equilibrium becomes unstable, giving rise to a stable periodic orbit.  This makes the dynamics of the second border collision \eqref{eq:v2} more globally interesting.  For convenience we will define
\begin{equation}
\label{eq:ep0}
\ep_0=-\dfrac{1}{15}.
\end{equation}

In the interest of facilitating the discussion of the global bifurcation, we will look at the boundary collision in a nonstandard direction, as a pseudoequilibrium which becomes real when it escapes the splitting manifold through an endpoint of a sliding region.  The following theorems address the existence and stability of the equilibrium and pseudoequilibrium as the system moves through the bifurcation.

\begin{theorem}
\label{thm:pseudonode}
For $\ep_0<\varepsilon<0$, there exists a unique unstable pseudonode.
\end{theorem}

\begin{proof}
Stability of the sliding region is given by 
\[
\ode{}{\lambda}(f\cdot\nabla h)|_{y=0}=\ode{}{k}(\beta-\beta\varepsilon-k\varepsilon-\alpha-(\beta+k)y-(\alpha\beta-\alpha)x)|_{y=0}=-\varepsilon.
\]
So, for $\varepsilon<0$, the sliding region is unstable.  We must show that the sliding region contains a unique pseudoequilibrium, and that pseudoequilibrium is unstable.  We first show that it is unique.

Pseudoequilibria are found on regions of the splitting manifold where the vector fields are colinear.  This gives a condition on $x$ and $\varepsilon$
\begin{equation}
\label{eq:pseudo}
-10\varepsilon-3x+5x\varepsilon+4x^2=0.
\end{equation}
So, we need to show that for $\ep_0<\ep<0$, there is only one solution to this equation inside the sliding region.  This result can be clearly seen by plotting equation \eqref{eq:pseudo} and the bounds of the sliding region, as in Figure \ref{fig:thm1pic}, however for the sake of rigor we prove uniqueness below.

We can rewrite equation \eqref{eq:pseudo} as follows
\begin{equation}
\label{eq:pseudoep}
\ep=\dfrac{4x^2-3x}{10-5x},
\end{equation}
Noting that $x=2$ was also not a valid solution in the original equation.  The $\ep$ range $\ep_0<\ep<0$ corresponds to a range of $x$ values, 
\[
\dfrac{1}{2}<x<\dfrac{3}{4}.
\]
Therefore, because the only minimum of equation \eqref{eq:pseudoep} occurs at 
\[
x=2-\dfrac{\sqrt{10}}{2}
\]
and the asymptote occurs at $x=2$, if the pseudoequilibrium exists, it is unique.  To show that the pseudoequilibrium exists, we employ a change of perspective, which is possible only because the bounds of the sliding region are lines in $(x,\ep)$ space.  First it is obvious that a solution to equation \eqref{eq:pseudoep} exists for $\ep\in\left(\ep_0,0\right)$.  The boundaries of the sliding region are given by lines in $(x,\ep)$ space, 
\[
x=\dfrac{3}{4}+\dfrac{15}{4}\ep
\]
and 
\[
x=\dfrac{3}{4}+\dfrac{5}{4}\ep.
\]
Writing these instead as $\ep$ bounds, we see that 
\[
\dfrac{3}{4}+\dfrac{15}{4}\ep<x<\dfrac{3}{4}+\dfrac{5}{4}\ep
\]
if and only if 
\[
\frac{4}{5}x-\frac{3}{5}<\ep<\frac{4}{15}x-\frac{1}{5}.
\]
Analysis shows that 
\[
\frac{4x^2-3x}{10-5x}>\frac{4}{5}x-\frac{3}{5}
\]
and 
\[
\frac{4x^2-3x}{10-5x}<\frac{4}{15}x-\frac{1}{5}
\]
as long as 
\[
\frac{1}{2}<x<\frac{3}{4},
\]
meaning that the $\ep$ given by equation \eqref{eq:pseudoep} is in the correct bounds, so $x$ is a pseudoequilibrium for the given $\ep$.

We now show that the pseudoequilibrium is an unstable equilibrium of the sliding flow.  The sliding flow is found by solving the system \eqref{eq:s}, and then using equation \eqref{eq:slidingflow}.  In this system, $\lambda^*(x)=k^*(x)$, and the flow in the sliding region is given by
\[
\dot{x}=1-x-k^*(x)x
\]
The stability of the pseudoequilibrium is then given by the sign of the derivative of the right side of this equation, evaluated at the equilibrium.  The derivative is 
\[
f'(x)=1-k^*-\left(\ode{k^*}{x}\right)x=\dfrac{-3+5\ep+8x}{-10\ep}.
\]
We've already seen that for these $\ep$ values, the pseudoequilibrium is bounded in $x$ by 
\[
\dfrac{1}{2}<x<\dfrac{3}{4}.
\]
This gives us a range on $f'$, 
\[
\dfrac{1+5\ep}{-10\ep}<f'(\ep)<\dfrac{3+5\ep}{-10\ep}.
\]
Both the upper and lower bounds are positive for $\ep_0<\ep<0$, implying that the pseudoequilibrium is unstable.  So, in this parameter range the pseudoequilibrium is totally unstable.

\begin{figure}[t]
	\centering
	{\includegraphics[width=0.5\textwidth]{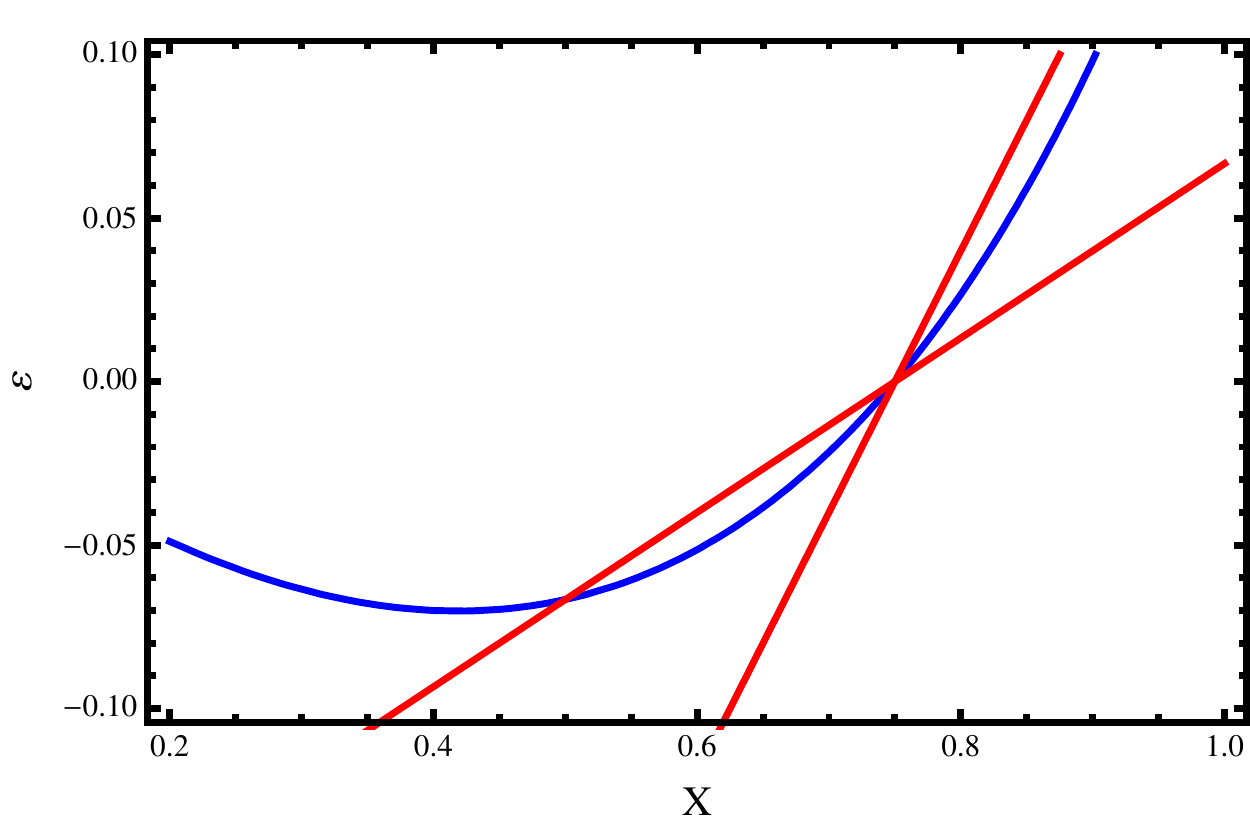}}
	
	\caption{Existence of a pseudoequilibrium of the sliding region.  Red lines represent the boundaries of the sliding region as functions of $x$.  The blue curve is the solution to equation \eqref{eq:pseudoep}.}
	\label{fig:thm1pic}
\end{figure}

\end{proof}

\begin{theorem}
\label{thm:stablebit}
For $\varepsilon<\ep_0$, there exists one globally stable, real equilibrium, and no pseudoequilibria.
\end{theorem}

\begin{proof}

First, the location of the equilibrium for the region $k=1$ is given by \eqref{eq:eq1}, which shows that this equilibrium is real in the specified parameter range.  The stability of the equilibrium is not dependent on $\varepsilon$, and has eigenvalues 
\[
\lambda=-2
\]
and 
\[
\lambda=-\dfrac{3}{2},
\]
indicating stability.  the equilibrium for the region $k=0$ remains virtual.  So, it only remains to prove that there exists no pseudoequilibria in this range.  This is done using the same method as in Theorem \ref{thm:pseudonode}, and again analysis shows that for $\ep<\ep_0$, the $\ep$ given by equation \eqref{eq:pseudoep} is outside the range of the sliding region.  So, no pseudoequilibrium exists in this parameter range.

\end{proof}

The local dynamics at $\varepsilon=\ep_0$, corresponding to the border collision, are unique to nonsmooth systems (See Figure \ref{fig:localpic}).  In the region $y>0$, all trajectories converge to the equilibrium.  The eigenvectors are both transverse to the splitting manifold, and do not depend on $\varepsilon$, or the location of the equilibrium.  However, in the region $y<0$, trajectories are repelled from the equilibrium, and are instead attracted to the virtual equilibrium of the $k=0$ system.  In the direction along the sliding region, the filippov flow takes solutions away from the equilibrium.  So, the equilibrium has two unstable directions, one along the splitting manifold, given by the Filippov flow, and the other transverse to the manifold, with infinite repulsion.  One might call this equilibrium a center of sorts, however the location of the stable and unstable directions topologically disallow this situation in a smooth system.  

However, this local picture is unsurprising in a nonsmooth system, and is in fact typical of a border collision.  What may be surprising is the change in stability of the equilibrium as the system passes through the bifurcation, which is intimately tied to the angle between the splitting manifold and the vector field in the $k=0$ region.  If the vector field in the $k=0$ region were vertical, the equilibrium would not persist as a psuedoequilibrium, and the structure of the bifurcation would be uninteresting.

\begin{figure}[t]
	\centering
	{\includegraphics[width=0.5\textwidth]{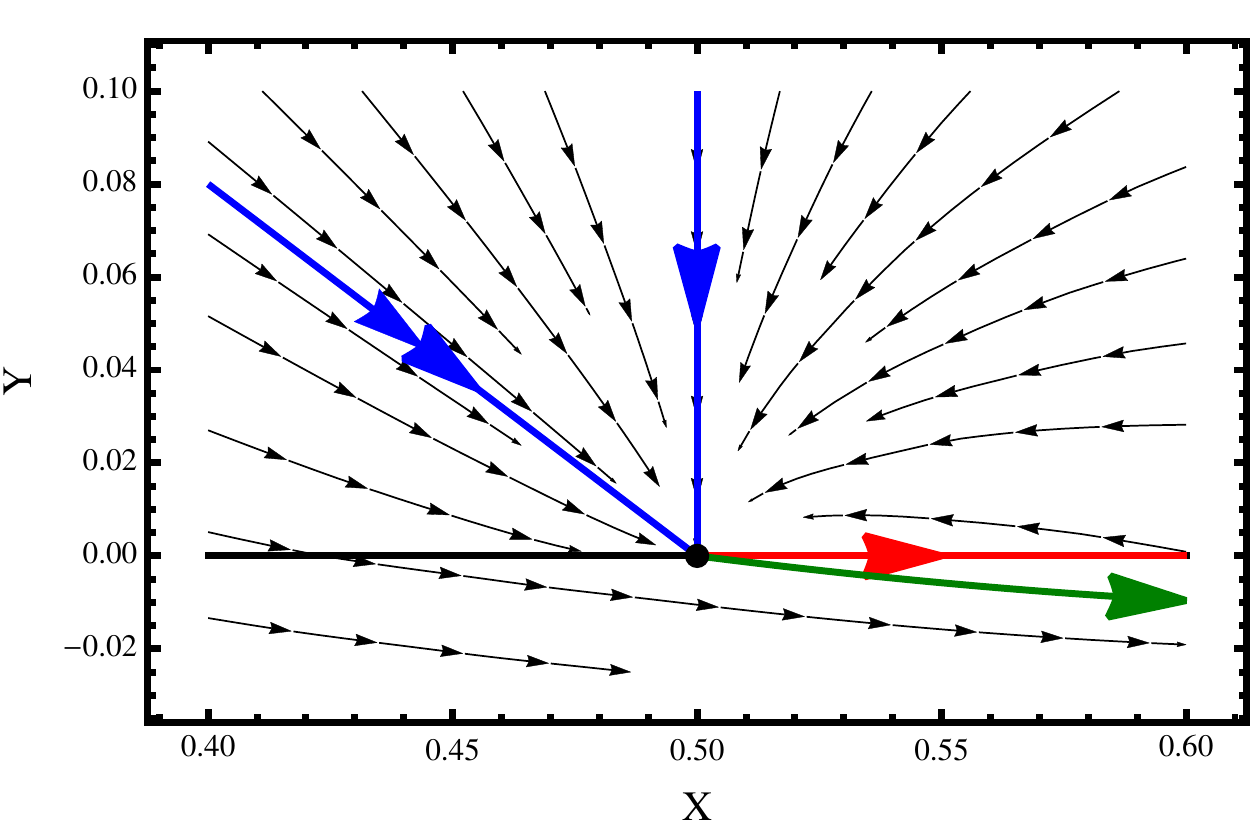}}
	
	\caption{Local stability of the boundary equilibrium.  Blue lines are in the direction of the eigenvectors.  The red line is the hyperbolic trajectory away from the equilibrium along the sliding region, and the green curve is the trajectory into the $y<0$ region, which leaves the equilibrium in finite time.}
	\label{fig:localpic}
\end{figure}

\section{Global Bifurcation Structure}
\label{sec:global}

The previous section discussed local dynamics as $\ep=\ep_0$, but the system also displays an interesting global bifurcation.  The globally stable periodic orbit transitions to a globally stable equilibrium point, through a homoclinic implosion of sorts.  The following theorems will explain this further.

\begin{theorem} For $\ep_0<\varepsilon<0$, there exists a periodic orbit which intersects the line $y=0$ with $x$ values in the interval
\[
I=\left(\dfrac{1}{2},\dfrac{3}{4}+\dfrac{15}{4}\varepsilon\right).
\]
\end{theorem}

\begin{proof}
We consider the return map of trajectories from $I$ to $I$, and show that the interval maps to a subset of itself.  Then, the Brouwer Fixed Point Theorem immediately gives the existence of a periodic orbit.  This is in fact very simple to show.  First, we note that trajectories with initial values in this interval return to the line $y=0$ in a bounded interval, with $x>0$.  Moreover, all trajectories originating from the splitting manifold with 
\[
x>\dfrac{3}{4}+\dfrac{15}{4}\varepsilon
\]
will intersect the splitting manifold again in $I$.  This is a result of uniqueness of solutions in smooth dynamical systems, and a lack of invariant sets in the region $y>0$.  The argument is three fold.  For 
\[
x>\dfrac{3}{4}+\dfrac{15}{4}\varepsilon, 
\]
and $\ep_0<\varepsilon<0$, 
\[
f(x,0)\cdot\begin{pmatrix} 0\\1\end{pmatrix}>0, 
\]
meaning the vector field points away from the splitting manifold, so trajectories can't return to the splitting manifold to the right of the interval.  The system in $y>0$ is also linear, and the equilibrium has a vertical eigenvector $\vec=e_2$, so we immediately conclude the existence of a vertical trajectory along the line 
\[
x=\dfrac{1}{2}, 
\]
converging to the virtual equilibrium.  So, trajectories can't return to the splitting manifold to the left of the interval.  However, solutions are easily seen to be bounded, and the lack of invariant sets in the region $y>0$ indicates that solutions must return to the splitting manifold, indicating that the return map exists, and maps $I$ to itself.  So, there exists a periodic orbit intersecting the splitting manifold in this interval.
\end{proof}

One should note that this result immediately implies that the periodic orbit limits to an orbit through 
\[
x=\dfrac{1}{2}.  
\]
This is in fact a homoclinic orbit coincident to the equilibrium border collision.

\begin{theorem} 
\label{thm:homo}
For $\varepsilon=\ep_0$, there exists a homoclinic orbit through 
\[
x=x_0=\frac{1}{2}.
\]
\end{theorem}
\begin{proof}
There is a unique trajectory emanating from the equilibrium point $\left(x_0,0\right)$ into the region $y<0$, based on the form of the $k=0$ system.  This trajectory crosses the splitting manifold at some point $x>1$.  All trajectories with initial condition $x>x_0$, $y>0$ converge to the equilibrium point without crossing the splitting manifold.  Trajectories are not allowed to cross the vector field at any point $x>x_0$, because in that region, $\dot{y}>0$ along $y=0$.  Moreover, trajectories can't cross the splitting manifold anywhere where $x<x_0$, because there exists a vertical line trajectory along $x=x_0$.  Therefore, trajectories are squeezed from either side, and in particular the unique trajectory in the $k=0$ region, with initial condition $x=x_0$ must then return to $x=x_0$ after an infinite period, giving a homoclinic orbit.
\end{proof}

\begin{figure}[t]
	\centering
	{\includegraphics[width=0.5\textwidth]{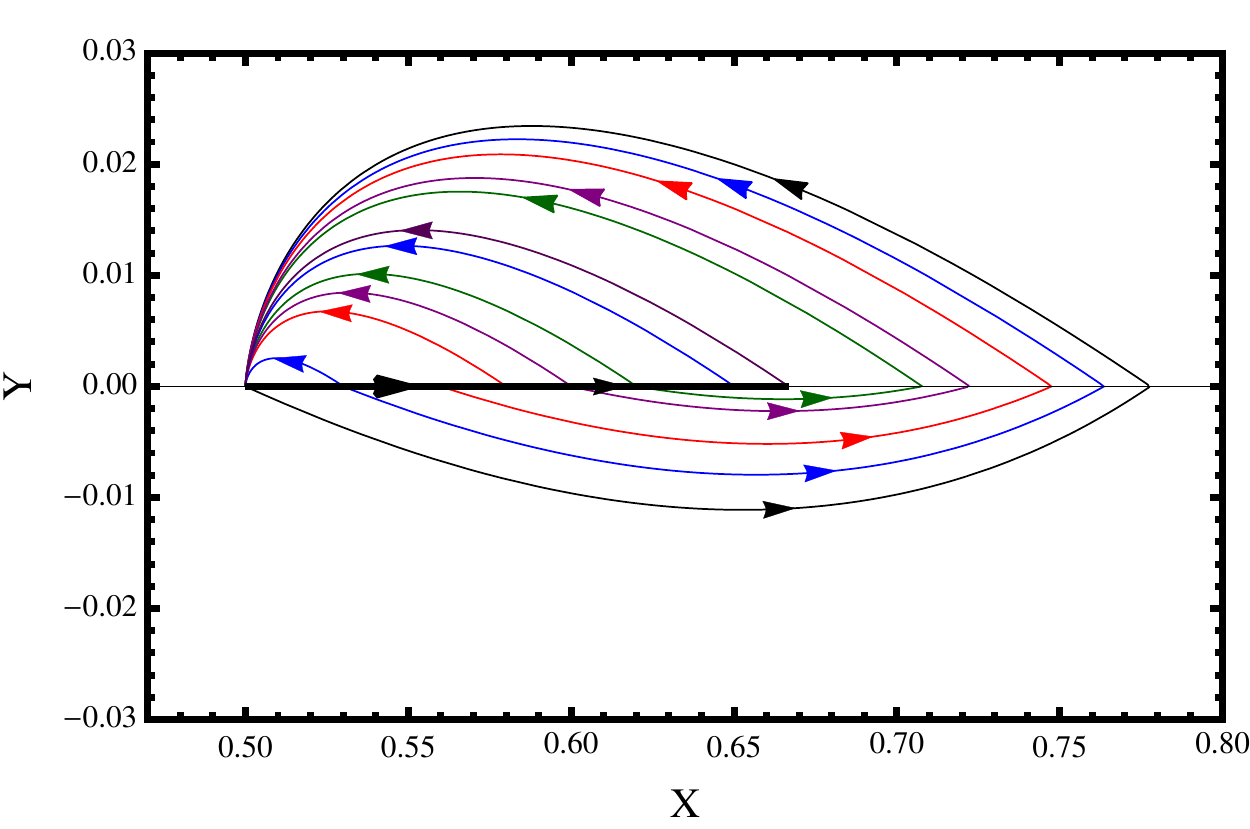}}
	
	\caption{Infinitely many homoclinic orbits occur at the bifurcation point.  Each orbit slides arbitrarily along the sliding region, before leaving in either direction to eventually converge to the boundary equilibrium.}
	\label{fig:homoclinicpic}
\end{figure}

\begin{theorem} Using Filippov dynamics, when $\ep=\ep_0$ all orbits which slide along the splitting manifold are homoclinic orbits.
\end{theorem}
\begin{proof}

First, because all trajectories in the $y>0$ region of phase space converge to the boundary equilibrium, it is clear that trajectories which travel arbitrarily along the sliding region and then escape to this region of phase space are homoclinic orbits.  If an orbit leaves the splitting manifold into the $y<0$ region, it is constrained by the homoclinic orbit described in Theorem \ref{thm:homo}.  Because there are no invariant sets in this region, these orbits will cross the splitting manifold, and then converge to the boundary equilibrium, and hence are also homoclinic orbits.  See Figure \ref{fig:homoclinicpic} for an illustration of these orbits.

\end{proof}

\section{The Corresponding Smooth Picture}
\label{sec:smooth}

\begin{figure}[t]
	\centering
	\subfloat[]{\includegraphics[width=0.3\textwidth]{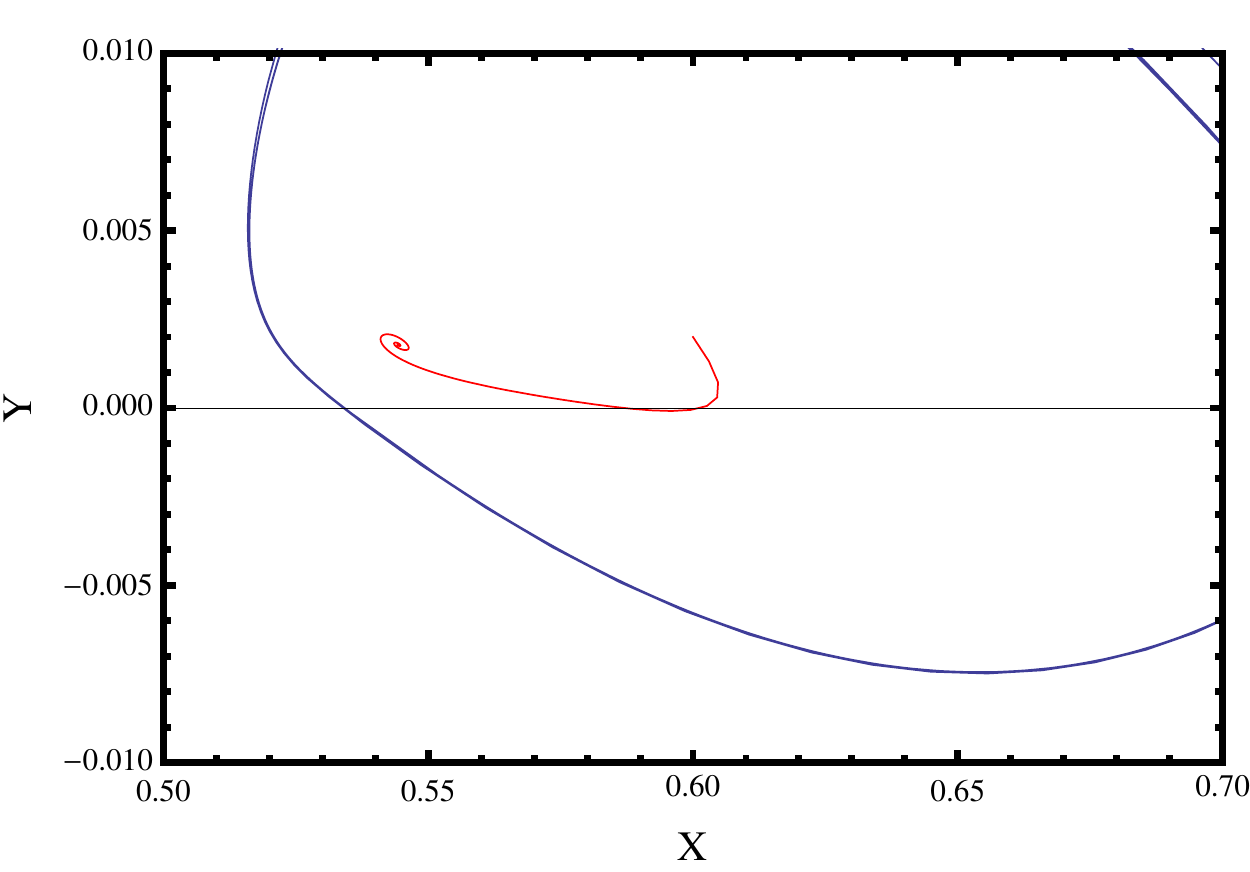}}
	\subfloat[]{\includegraphics[width=0.3\textwidth]{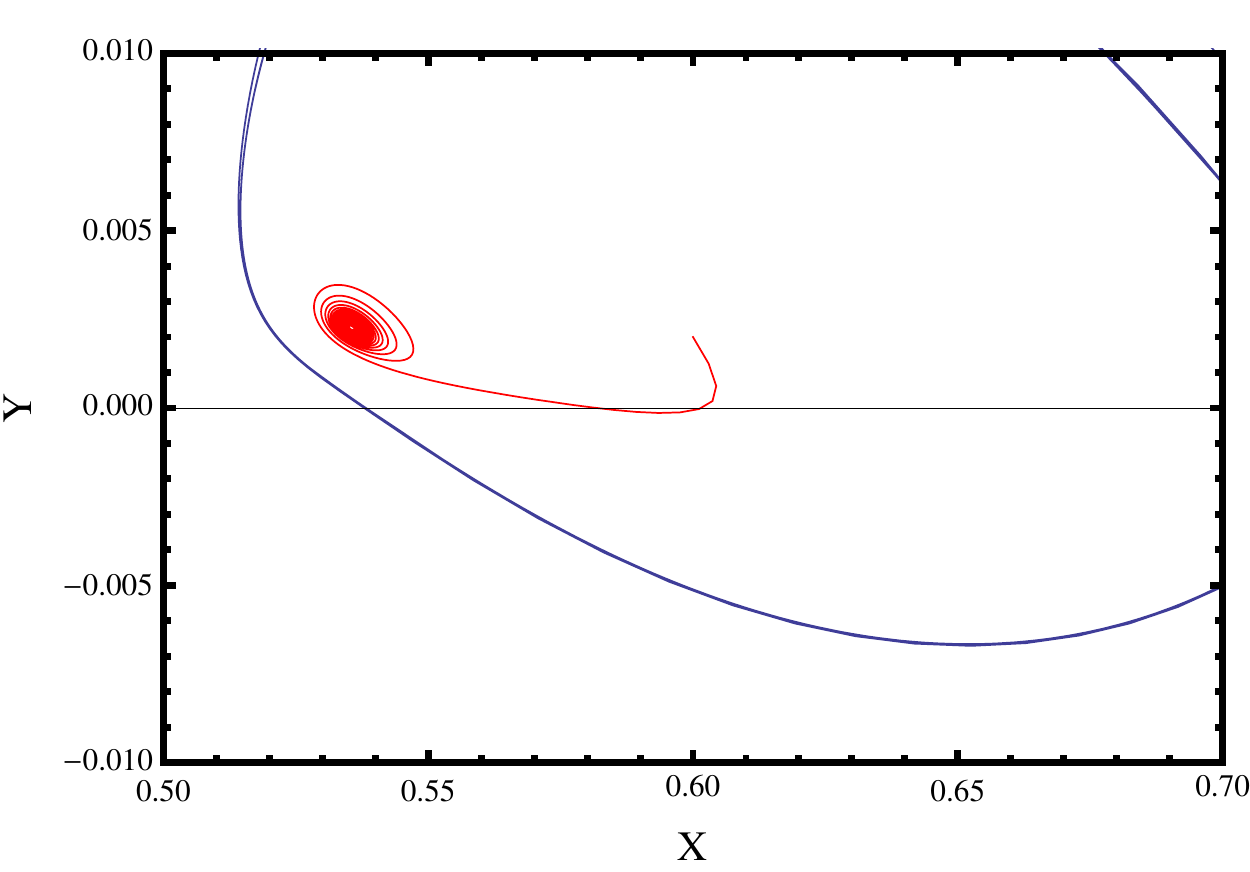}}
	\subfloat[]{\includegraphics[width=0.3\textwidth]{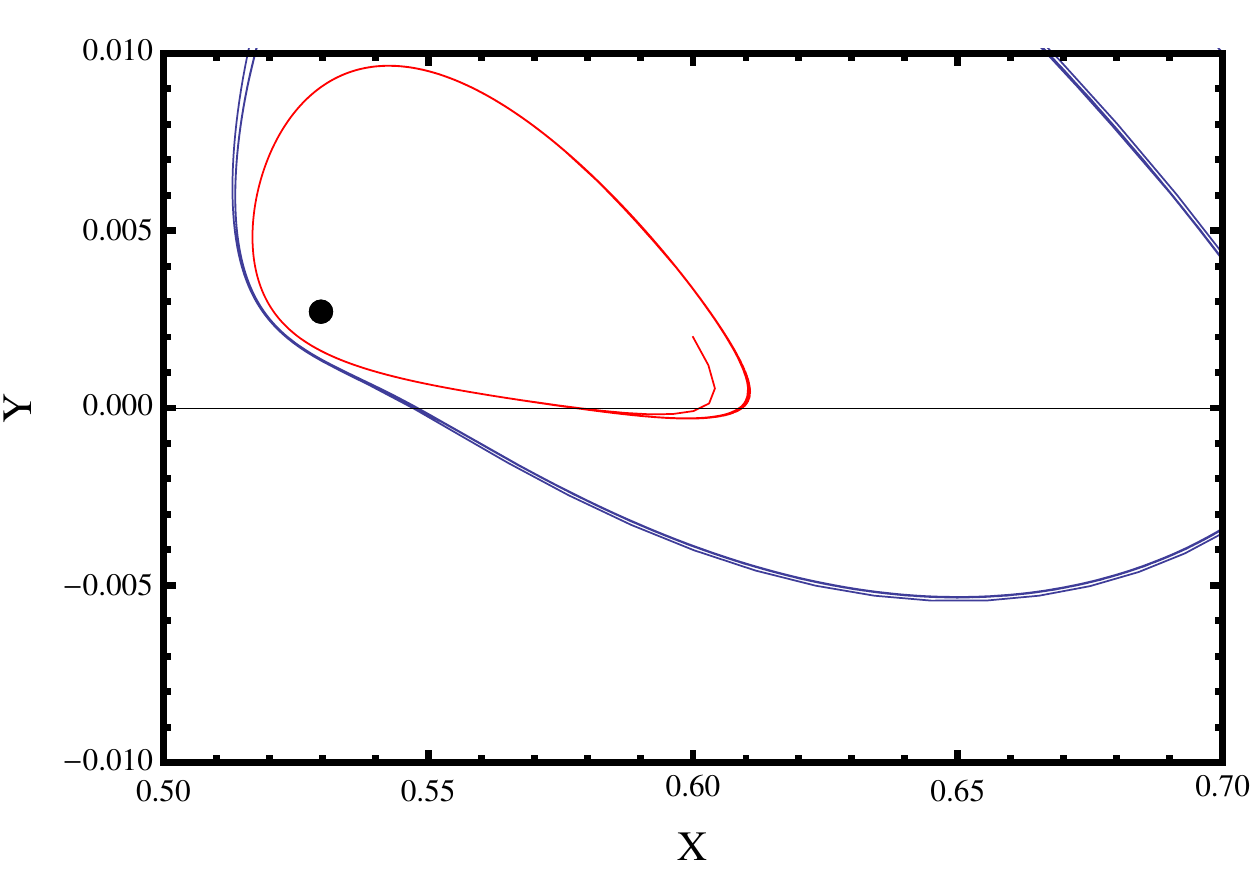}}\\
	\subfloat[]{\includegraphics[width=0.3\textwidth]{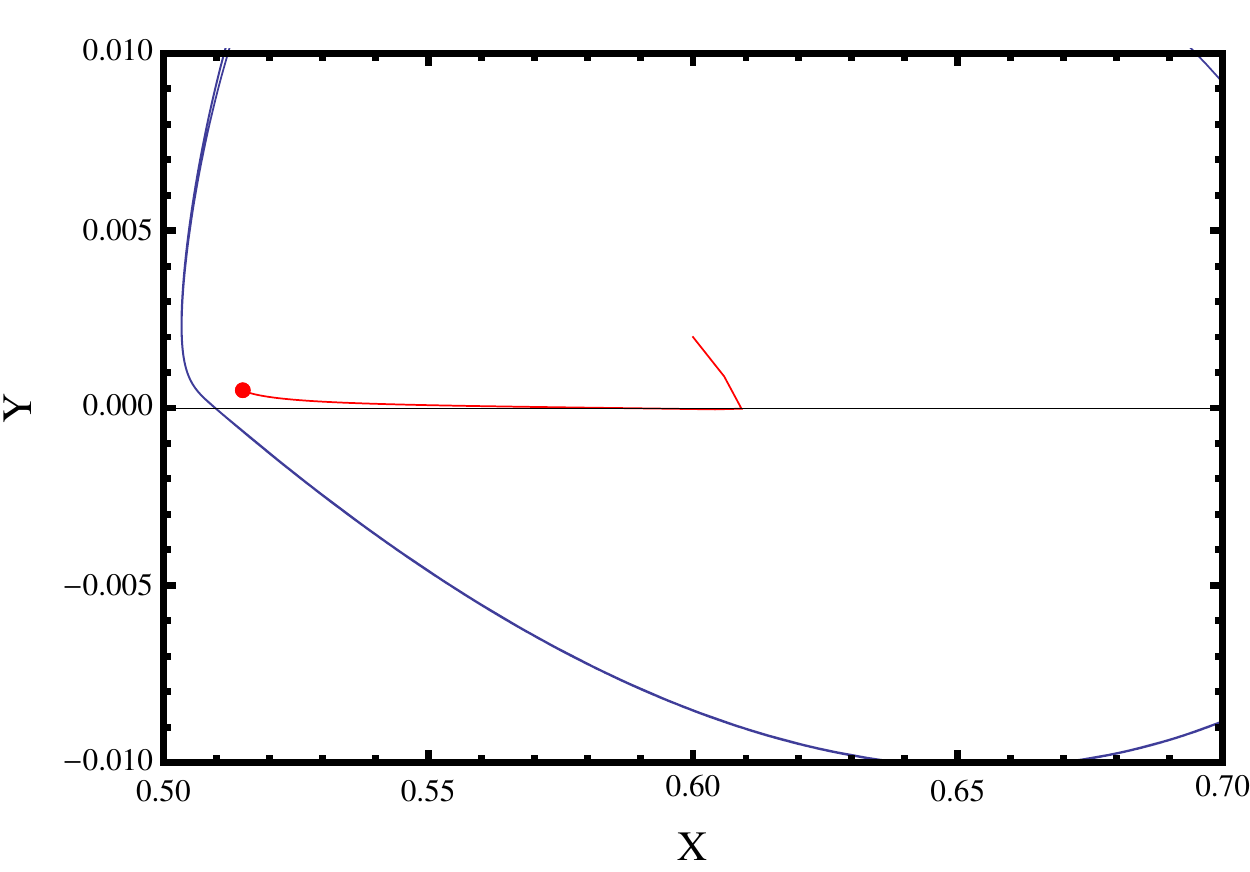}}
	\subfloat[]{\includegraphics[width=0.3\textwidth]{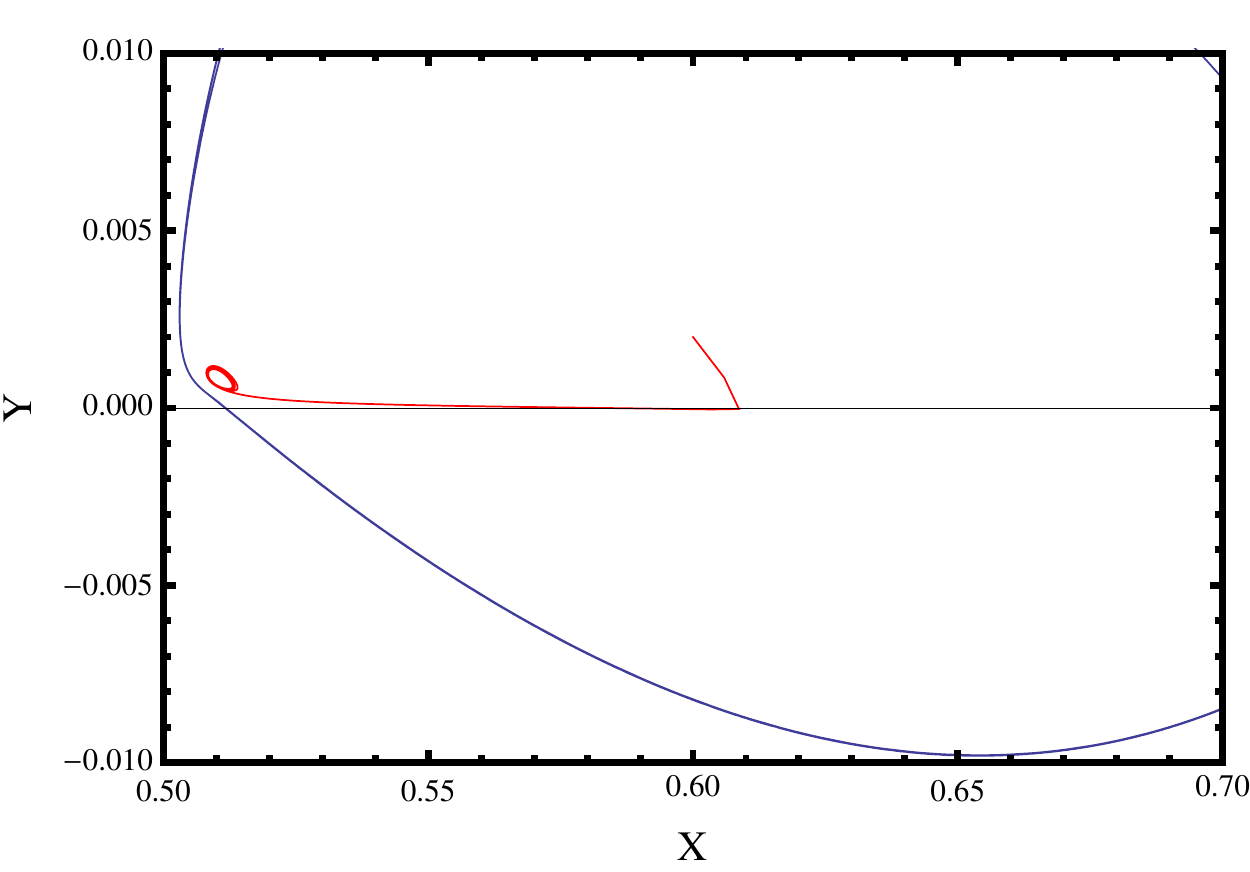}}
	\subfloat[]{\includegraphics[width=0.3\textwidth]{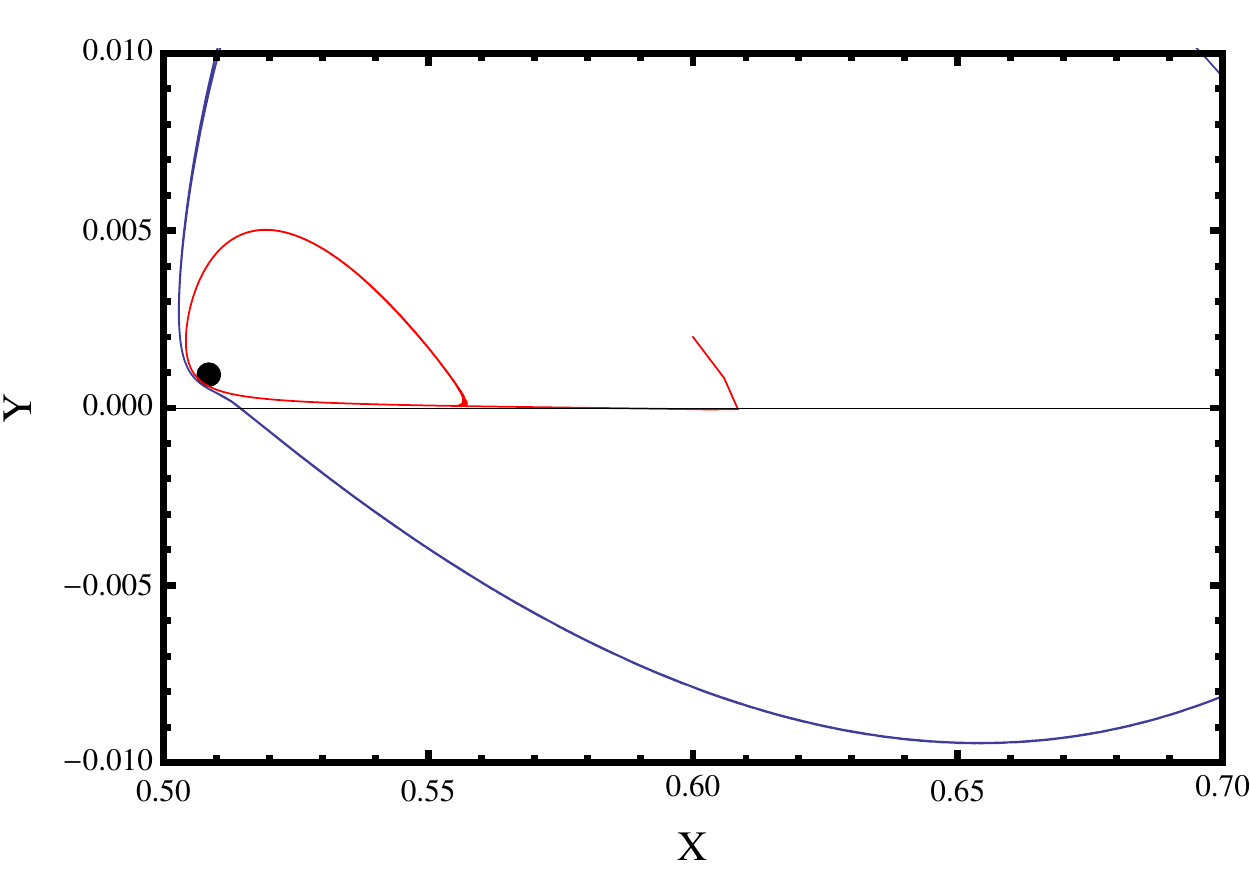}}
	\caption{The smooth system as it passes through the bifurcation.  A subcritical Hopf bifurcation and a periodic orbit saddle node bifurcation limit to a single bifurcation point as $a\rightarrow0$.  Top row: $a=10^{-3}$, A: $\ep=-0.0633$ B: $\ep=-0.0650$ C: $\ep=-0.0662$.  Bottom row: $a=10^{-4}$, D: $\ep=-0.0658$ E: $\ep=-0.0665$ F: $\ep=-0.0668$. }
	\label{fig:smooth}
\end{figure}

There are several key differences between this nonsmooth homoclinic bifurcation and a standard homoclinic bifurcation, which indicate that this phenomenon is not a limit of a smooth homoclinic bifurcation.  The bifurcation does not leave behind a saddle equilibrium point.  The homoclinic orbits are space filling.  Most importantly, there is a shear in the stable and unstable directions of the equilibrium on the homoclinic orbit.  Alternatively, the system shows some similarities to canard orbits.  The homoclinic orbits are of all sizes, as if there is a canard explosion occurring at a single point.  Moreover, the aftermath of the nonsmooth bifurcation is attraction to a larger periodic structure, analogous to a relaxation oscillation.  However, the model doesn't have the bistability inherent to canard systems.  So, one wonders which smooth phenomena might limit to this nonsmooth homoclinic bifurcation.

In this system, we can answer this question directly, because Welander formulated two versions of his model, the nonsmooth system being a the pointwise limit of a smooth system.  In the smooth system, instead of $k$ being a nonsmooth Heaviside function, 
\[
k=\frac{1}{\pi}\tan^{-1}\left(\frac{y}{a}\right)+\frac{1}{2}.
\]
Therefore, we can describe which behavior in the smooth model leads to such unusual behavior in the nonsmooth model.  Proving that the smooth phenomena limit to the homoclinic bifurcation is nontrivial, however simulations indicate that this is the case.

The smooth system contains a supercritical Hopf bifurcation which is a direct analogue to a similar phenomenon in the nonsmooth system.  However, the smooth system also contains a subcritical Hopf bifurcation in the parameter vicinity of $\varepsilon=\ep_0$.  The stable and unstable periodic orbits from these bifurcations grow until they annihilate each other in a periodic orbit saddle node bifurcation, leaving behind a globally stable equilibrium point.  The subcritical Hopf bifurcation and the periodic orbit saddle node limit to the same parameter value as $a\rightarrow0$, and the system becomes nonsmooth.

\section{Discussion}
\label{sec:discussion}

There are a few interesting implications of this result.  First, it is a testament to the complexities possible in nonsmooth systems, even in the simple case of equilibria colliding with a splitting manifold.  Consideration of many systems is necessary before any generalizations can be made about classification of behavior in these systems.  There are several interesting codimension one phenomena in Welander's model, however, a full bifurcation analysis of the model has not been undertaken.  It is possible that with inclusion of the other two model parameters the system contains bifurcations like a Takens Bogdanov, which contains some of the features seen in this analysis.  More analysis is necessary to determine if this is the case.  Moreover, a more general relationship between smooth and nonsmooth systems is not clear.  In this system, the nonsmooth structure is qualitatively different from the smooth structure.  In fact, even slightly away from the bifurcation point, the periodic orbit which is unique in the nonsmooth system does not perturb to a unique periodic orbit in the smooth system.  Instead, the smooth system has two periodic orbits, one stable and one unstable.  This is in spite of the nice (monotone) behavior of the smooth approximation to the Heaviside function.  It is possible that the homoclinic explosion is a fundamentally nonsmooth phenomenon, which is the limiting behavior of many different smooth behaviors.  So, this is yet another indication that care must be taken when using nonsmooth approximations to smooth systems.

\section*{Acknowledgments}
The author would like to thank Richard McGehee, Mike Jeffrey, Paul Glendinning, and Mary Silber for their input and advice on this work.

\nocite{*}
\newpage
\bibliographystyle{amsplain}
\bibliography{researchbib}

\end{document}